\documentclass[11pt]{article}

\usepackage{amsfonts,amsmath,latexsym,color,epsfig}
\newtheorem{theorem}{Theorem}[section]
\newtheorem{lemma}{Lemma}[section]

\newtheorem{claim}{Claim}

\newenvironment{proof}
      {\medskip\noindent{\bf Proof:}\hspace{1mm}}
      {\hfill$\Box$\medskip}
\def\qed{\ifvmode\mbox{ }\else\unskip\fi\hskip 1em plus 10fill$\Box$}

\input{epsf}

\makeatletter
\def\Ddots{\mathinner{\mkern1mu\raise\p@
\vbox{\kern7\p@\hbox{.}}\mkern2mu
\raise4\p@\hbox{.}\mkern2mu\raise7\p@\hbox{.}\mkern1mu}}
\makeatother

\title{\vspace{-0.7cm}On two problems in graph Ramsey theory}
\author{David Conlon\thanks{St John's College, Cambridge, United Kingdom.
E-mail: {\tt
D.Conlon@dpmms.cam.ac.uk}. Research supported by a Junior Research Fellowship
at St John's College.} \and Jacob Fox\thanks{Department of
Mathematics, Princeton, Princeton, NJ. Email: {\tt
jacobfox@math.princeton.edu}. Research supported by an NSF Graduate
Research Fellowship and a Princeton Centennial Fellowship.} \and
Benny Sudakov\thanks{Department of Mathematics,
UCLA,  Los Angeles, CA 90095. Email: {\tt bsudakov@math.ucla.edu}. Research
supported in part by NSF CAREER award DMS-0812005 and by a
USA-Israeli BSF grant.}}
\date{}
\begin{document}
\maketitle

\begin{abstract}
We study two classical problems in graph Ramsey theory, that of determining the
Ramsey number of bounded-degree graphs and that of estimating the induced
Ramsey number for a graph with a given number of vertices.

The Ramsey number $r(H)$ of a graph $H$ is the least positive integer $N$ such
that every two-coloring of the edges of the complete graph $K_N$ contains a
monochromatic copy of $H$. A famous result of Chv\'atal, R\"{o}dl, Szemer\'edi
and Trotter states that there exists a constant $c(\Delta)$ such that $r(H)
\leq c(\Delta) n$ for every graph $H$ with $n$ vertices and maximum degree
$\Delta$. The important open question is to determine the constant $c(\Delta)$.
The best results, both due to Graham, R\"{o}dl and Ruci\'nski, state that there
are constants $c$ and $c'$ such that $2^{c' \Delta} \leq c(\Delta) \leq 2^{c
\Delta \log^2 \Delta}$. We improve this upper bound, showing that there is a
constant $c$ for which $c(\Delta) \leq 2^{c \Delta \log \Delta}$.

The induced Ramsey number $r_{ind}(H)$ of a graph $H$ is the least positive
integer $N$ for which there exists a graph $G$ on $N$ vertices such that every
two-coloring of the edges of $G$ contains an induced monochromatic copy of $H$.
Erd\H{o}s conjectured the existence of a constant $c$ such that, for any graph
$H$ on $n$ vertices, $r_{ind}(H) \leq 2^{c n}$. We move a step closer to
proving this conjecture, showing that $r_{ind} (H) \leq 2^{c n \log n}$. This
improves upon an earlier result of Kohayakawa, Pr\"{o}mel and R\"{o}dl by a
factor of $\log n$ in the exponent.
\end{abstract}

\section{Introduction}

Given a graph $H$, the {\it Ramsey number} $r(H)$ is defined to be the smallest
natural number $N$ such that, in any two-coloring of the edges of $K_N$, there
exists a monochromatic copy of $H$. That these numbers exist was first proven
by Ramsey \cite{R30} and rediscovered independently by Erd\H{o}s and Szekeres
\cite{ErSz}. Since their time, and particularly since the 1970s, Ramsey theory
has grown into one of the most active areas of research within combinatorics,
overlapping variously with graph theory, number theory, geometry and logic.

The most famous question in the field is that of estimating the Ramsey number
$r(t)$ of the complete graph $K_t$ on $t$ vertices. However, despite some small
improvements \cite{S75, Co}, the standard estimates, that $2^{t/2} \leq r(t)
\leq 2^{2t}$, have remained largely unchanged for over sixty years.
Unsurprisingly
then, the field has stretched in different directions. One such direction that
has become fundamental in its own right is that of looking at what happens to
the Ramsey number when we are dealing with various types of sparse graphs.
Another is that of determining induced Ramsey numbers, i.e., proving, for
any given $H$, that there is a graph $G$ such that any two-coloring of the
edges of $G$ contains an induced monochromatic copy of $H$. In this paper, we
present a unified approach which allows us to make improvements to two
classical questions in these areas.

In 1975, Burr and Erd\H{o}s \cite{BE75} posed the problem of showing that every
graph $H$ with $n$ vertices and maximum degree $\Delta$ satisfied $r(H) \leq
c(\Delta) n$,
where the constant $c(\Delta)$ depends only on $\Delta$. That this is indeed
the case was shown by Chv\'atal, R\"{o}dl, Szemer\'edi and Trotter
\cite{CRST83} in one of the earliest applications of Szemer\'edi's celebrated
regularity lemma \cite{Sz76}. Remarkably, this means that for graphs of fixed
maximum degree the Ramsey number only has a linear dependence on the number of
vertices. Unfortunately, because it uses the regularity lemma, the bounds that
the original method gives on $c(\Delta)$ are (and are necessarily \cite{G97})
of tower type in $\Delta$. More precisely, $c(\Delta)$ works out as being an
exponential tower of 2s with a height that is itself exponential in $\Delta$.

The situation was remedied somewhat by Eaton \cite{E98}, who proved, using a
variant of the regularity lemma, that the function $c(\Delta)$ can be taken to
be of the form $2^{2^{c \Delta}}$. Soon after, Graham, R\"{o}dl and Ruci\'nski
proved \cite{GRR00}, by a beautiful method which avoids any use of the
regularity lemma, that there exists a constant $c$ for which
\[c(\Delta) \leq 2^{c \Delta \log^2 \Delta}.\]
For bipartite graphs, they were able to do even better \cite{GRR01}, showing
that if $H$ is a bipartite graph with $n$ vertices and maximum degree $\Delta$
then $r(H) \leq 2^{c \Delta \log \Delta} n$. They also proved that there are
bipartite graphs with $n$ vertices and maximum degree $\Delta$ for which the
Ramsey number is at least $2^{c' \Delta} n$. Recently, Conlon \cite{C07} and,
independently, Fox and Sudakov \cite{FS07} have shown how to remove the $\log
\Delta$ factor in the exponent, achieving an essentially best possible bound of
$r(H) \leq 2^{c \Delta} n$ in the bipartite case. These results were jointly
extended to hypergraphs in \cite{CFS09}, after several proofs
\cite{CFKO07,CFKO072,NORS07} using the hypergraph regularity lemma.

Unfortunately, if one tries to use these recent techniques to treat general
graphs, the best one seems to be able to achieve is $c(\Delta) \leq 2^{c
\Delta^2}$. In this paper we take a different approach, more closely related to
that of Graham, R\"{o}dl and Ruci\'nski \cite{GRR00}. Improving on their bound,
we show that $c(\Delta) \leq 2^{c \Delta \log \Delta}$, which brings us a step
closer to matching the lower bound of $2^{c' \Delta}$.

\begin{theorem} \label{MaxDegreeIntro}
There exists a constant $c$ such that, for every graph $H$ with $n$ vertices
and maximum degree $\Delta$,
\[r(H) \leq 2^{c \Delta \log \Delta} n.\]
\end{theorem}

A graph $H$ is said to be an {\it induced subgraph} of $H$ if $V(H) \subset
V(G)$ and two vertices of $H$ are adjacent if and only if they are adjacent in
$G$. The {\it induced Ramsey number} $r_{ind} (H)$ is the smallest natural
number $N$ for which there is a graph $G$ on $N$ vertices such that in every
two-coloring of the edges of $G$ there is an induced monochromatic copy of $H$.
The existence of these numbers was independently proven by Deuber \cite{De},
Erd\H{o}s, Hajnal and P\'osa \cite{ErHaPo} and R\"{o}dl \cite{Ro1}. The bounds
that these original proofs give on $r_{ind} (H)$ are enormous, but it was
conjectured by Erd\H{o}s \cite{Er2} that the actual values should be more in
line with ordinary Ramsey numbers. More specifically, he conjectured the
existence of a constant $c$ such that every graph $H$ with $n$ vertices
satisfies $r_{ind} (H) \leq 2^{c n}$. If true, the complete graph shows that it
would be best possible.

In a problem paper, Erd\H{o}s \cite{Er1} stated that he and Hajnal had proved a
bound of the form $r_{ind} (H) \leq 2^{2^{n^{1 + o(1)}}}$. This remained the
state of the art for some years until Kohayakawa, Pr\"{o}mel and R\"{o}dl
\cite{KoPrRo} proved that there was a constant $c$ such that every graph $H$ on
$n$ vertices satisfies $r_{ind} (H) \leq 2^{c n \log^2 n}$. As in the
bounded-degree problem, we remove one of the logarithms in the exponent.

\begin{theorem} \label{InducedIntro}
There exists a constant $c$ such that every graph $H$ with $n$ vertices
satisfies
\[r_{ind} (H) \leq 2^{c n \log n}.\]
\end{theorem}

It is worth noting that the graph $G$ that Kohayakawa, Pr\"{o}mel and R\"{o}dl
use in their proofs is a random graph constructed with projective planes. This
graph is specifically designed so as to contain many copies of our target graph
$H$. Recently, Fox and Sudakov \cite{FS08} showed how to prove the same bounds
as Kohayakawa, Pr\"{o}mel and R\"{o}dl using explicit pseudo-random graphs. We
will follow a similar path.

A graph is said to be pseudo-random if it imitates some of the properties of a
random graph. One such random-like property, introduced by Thomason \cite{T87,
T872}, is that of having approximately the same density between any pair of
large disjoint vertex sets. More formally, we say that a graph $G = (V, E)$ is
{\it $(p, \lambda)$-pseudo-random} if, for all subsets $A, B$ of $V$, the
density of edges $d(A,B)$ between $A$ and $B$ satisfies
\[|d(A, B) - p| \leq \frac{\lambda}{\sqrt{|A||B|}}.\]
The usual random graph $G(N,p)$, where each edge is chosen independently with
probability $p$, is itself a $(p, \lambda)$-pseudo-random graph where $\lambda$
is on the order of $\sqrt{N}$. A well-known explicit example, known to be
$(\frac{1}{2}, \sqrt{N})$-pseudo-random, is the Paley graph $P_N$. This graph
is defined by setting $V$ to be the set $\mathbb{Z}_N$, where $N$ is a prime
which is congruent to 1 modulo 4, and taking two vertices $x, y \in V$ to be
adjacent if and only if $x - y$ is a quadratic residue. For further information
on this and other pseudo-random graphs we refer the reader to \cite{KS06}. Our
next theorem states that, for $\lambda$ sufficiently small, a $(\frac{1}{2},
\lambda)$-pseudo-random graph has very strong Ramsey properties. Theorem
\ref{InducedIntro} follows by applying this theorem to the particular examples
of pseudo-random graphs given above.

\begin{theorem} \label{maininduced}
There exists a constant $c$ such that, for any $n \in \mathbb{N}$ and any
$(\frac{1}{2}, \lambda)$-pseudo-random graph $G$ on $N$ vertices with $\lambda
\leq 2^{-c n \log n} N$, every graph on $n$ vertices occurs as an induced
monochromatic copy in all 2-edge-colorings of $G$. Moreover, all of these
induced monochromatic copies can be found in the same color.
\end{theorem}

The theme that unites these two, apparently disparate, questions is the method
we employ in our proofs. A simplified version of this method is the following.
In the first color we attempt to find a large subset in which this color is
very dense. If such a set can be found, we can easily embed the required graph.
If, on the other hand, this is not the case, then there is a large subset in
which the edges of the second color are
well-distributed. Again, this allows us to prove an embedding lemma. Such ideas
are already explicit in the work of Graham,
R\"{o}dl and Ruci\'nski and, arguably, implicit in that of Kohayakawa,
Pr\"{o}mel and R\"{o}dl. The advantage of our method, which extends upon these
ideas, is that it is much more symmetrical between the colors. It is this
symmetry which allows us
to drop a log factor in each case.

In the next section, we will prove Theorem \ref{MaxDegreeIntro}. Section
\ref{InducedSection} contains the proof of Theorem \ref{maininduced}. The last
section contains some concluding remarks together with a discussion of a few
conjectures and open problems. Throughout the paper, we systematically omit
floor and ceiling signs whenever they are not crucial for the sake of clarity
of presentation. All logarithms, unless otherwise stated, are to the base $2$.
We also do not make any serious attempt to optimize absolute constants in our
statements and proofs.

\section{Ramsey number of bounded-degree graphs} \label{BoundedSection}

The {\it edge density} $d(X,Y)$ between two disjoint vertex subsets $X,Y$ of a
graph $G$ is the fraction of pairs $(x,y)
\in X \times Y$ that are edges of $G$. That is,
$d(X,Y)=\frac{e(X,Y)}{|X||Y|}$, where $e(X,Y)$ is the number of edges with one
endpoint in $X$ and
the other in $Y$. In a graph $G$, a vertex subset $U$ is called {\it
bi-$(\epsilon,\rho)$-dense} if, for all disjoint pairs $A,B \subset U$ with
$|A|,|B| \geq \epsilon |U|$, we have $d(A,B) \geq \rho$. We call a graph $G$
{\it bi-$(\epsilon,\rho)$-dense} if its vertex set $V(G)$ is
bi-$(\epsilon,\rho)$-dense. Trivially, if $\epsilon' \leq \epsilon$ and $\rho'
\geq \rho$, then a bi-$(\epsilon',\rho')$-dense graph is also
bi-$(\epsilon,\rho)$-dense. Moreover, if $\epsilon>1/2$, then every graph is
vacuously bi-$(\epsilon,\rho)$-dense as there is no pair of disjoint subsets
each with more than half of the vertices.

Before going into the proof of Theorem \ref{MaxDegreeIntro}, we first sketch
for
comparison the original idea of Graham, R\"odl, and Rucinski \cite{GRR00} which
gives a weaker bound. We then discuss our proof technique. They noticed that if
a graph $G$ on $N$ vertices is bi-$(\epsilon,\rho)$-dense with
$\epsilon=\rho^{\Delta}/(\Delta+1)$ and $N \geq 2\rho^{-\Delta}(\Delta+1)n$,
then $G$ contains every $n$-vertex graph $H$ of maximum degree $\Delta$. This
can be shown by embedding $H$ one vertex at a time. In particular, if a
red-blue edge-coloring of $K_N$ does not contain a monochromatic copy of $H$,
then the red graph is not bi-$(\epsilon,\rho)$-dense, and there are disjoint
vertex subsets $A$ and $B$ with $|A|,|B| \geq \epsilon N$ such that the red
density between them at most $\rho$. It is then possible to iterate, at the
expense of another factor in the exponent of roughly $\log (1/\rho)$, to get a
subset $S$ of size roughly $\epsilon^{\log (1/\rho)}N$ with red edge density at
most $2\rho$ inside. Picking $\rho = \frac{1}{16\Delta}$, a simple greedy
embedding then shows that inside $S$ we can find a blue copy of any graph with
at
most $|S|/4$ vertices and maximum degree $\Delta$.

To summarize, the proof finds a vertex subset $S$ which is either
bi-$(\epsilon,\rho)$-dense in the red graph or is very dense in the blue graph.
In either case, it is easy to find a monochromatic copy of any $n$-vertex graph
$H$ with maximum degree $\Delta$.

We will instead find a sequence of large vertex subsets $S_1,\ldots,S_t$ such
that, in one of the two colors, each of the subsets satisfies some bi-density
condition and the graph between these subsets is very dense.  The bi-density
condition inside each $S_i$ is roughly the condition which ensures that we can
embed any graph on $n$ vertices with maximum degree $d_i$, where
$d_1+\ldots+d_t=\Delta-t+1$. A simple lemma of Lov\'asz guarantees that we can
partition $V(H)=V_1 \cup \ldots \cup V_t$ such that the induced subgraph of $H$
with vertex set $V_i$ has maximum degree at most $d_i$. Our embedding lemma
shows that we can embed a monochromatic copy of $H$ with the image of $V_i$
being in
$S_i$. We now proceed to the details of the proof.

\begin{definition} A graph on $N$ vertices is {\it
$(\alpha,\beta,\rho,\Delta)$-dense} if there is a sequence $S_1,\ldots,S_t$ of
disjoint vertex subsets each of cardinality at least $\alpha N$ and nonnegative
integers $d_1,\ldots,d_t$ such that $d_1+\cdots+d_t=\Delta-t+1$, and the
following holds:
\begin{itemize}
\item for $1 \leq i \leq t$, $S_i$ is bi-$(\rho^{2d_i},\rho)$-dense, and
\item for $1 \leq i < j \leq t$, each vertex in $S_i$ has at least
$(1-\beta)|S_j|$ neighbors in $S_j$.
\end{itemize}
\end{definition}

Note that since $d_1+\cdots+d_t=\Delta-t+1$ and each $d_i$ is
nonnegative, we must have $t \leq \Delta+1$.

Trivially, if a graph is $(\alpha',\beta',\rho,\Delta')$-dense and $\alpha'
\geq \alpha$, $\beta' \leq \beta$, and $\Delta' \geq \Delta$,
then it is also  $(\alpha,\beta,\rho,\Delta)$-dense.

We say a red-blue edge-coloring of the complete graph $K_N$ is {\it
$(\alpha,\beta,\rho,\Delta_1,\Delta_2)$-dense} if the red graph is
$(\alpha,\beta,\rho,\Delta_1)$-dense or the blue graph is
$(\alpha,\beta,\rho,\Delta_2)$-dense. We say that
$(\alpha,\beta,\rho,\Delta_1,\Delta_2)$ is {\it universal} if, for every $N$,
every red-blue edge-coloring of $K_N$ is
$(\alpha,\beta,\rho,\Delta_1,\Delta_2)$-dense.

\begin{lemma}\label{firstlemma}
If $\beta \geq 4(\Delta_2+1)\rho$ and $(\alpha,\beta,\rho,\Delta_1,\Delta_2)$
is universal, then
$(\frac{1}{2}\rho^{2\Delta_1}\alpha,\beta,\rho,\Delta_1,2\Delta_2+1)$ is also
universal.
\end{lemma}
\begin{proof}
Consider a red-blue edge-coloring of a complete graph $K_N$. If the red graph
is bi-$(\rho^{2\Delta_1},\rho)$-dense, then, taking $t=1$, $S_1=V(K_N)$ and
$d_1=\Delta_1$, we see that the red graph is
$(\alpha,\beta,\rho,\Delta_1)$-dense and we are done. So we may suppose that
there are disjoint vertex subsets $V_0,V_1$ with $|V_0|,|V_1| \geq
\rho^{2\Delta_1}N$ such that the red density between them is less than $\rho$.
Delete from $V_0$ all vertices in at least $2\rho|V_1|$ red edges with vertices
in $V_1$; the remaining subset $V_0'$ has cardinality at least
$\frac{1}{2}|V_0| \geq \frac{1}{2}\rho^{2\Delta_1}N$. Since $(\alpha, \beta,
\rho, \Delta_1, \Delta_2)$ is universal, the coloring restricted to $V_0'$ is
$(\alpha,\beta,\rho,\Delta_1,\Delta_2)$-dense. Thus, the
red graph is $(\alpha,\beta,\rho,\Delta_1)$-dense (in which case we are again
done)
or the blue graph is $(\alpha,\beta,\rho,\Delta_2)$-dense. We may suppose the
latter holds, and there are subsets $S_1,\ldots,S_t$ each of cardinality at
least $\alpha |V_0'| \geq \frac{1}{2}\rho^{2\Delta_1}\alpha N$ and nonnegative
integers $d_1,\ldots,d_t$ such that $d_1+\cdots+d_t=\Delta_2-t+1$, and the
following holds:
\begin{itemize}
\item for $1 \leq i \leq t$, $S_i$ is bi-$(\rho^{2d_i},\rho)$-dense, and
\item for $1 \leq i < j \leq t$, each vertex in $S_i$ has at least
$(1-\beta)|S_j|$ neighbors in $S_j$.
\end{itemize}

Since each vertex in $V_0'$ (and hence in each $S_i$) is in at most
$2\rho|V_1|$ red edges with vertices in $V_1$, there are
at most $2\rho|S_i||V_1|$ red edges between $S_i$ and $V_1$. For $1\leq i \leq
t$, delete from $V_1$ all vertices
in at least $4(\Delta_2+1)\rho|S_i|$ red edges with vertices in $S_i$. For any
given $i$, there can be at most
$\frac{1}{2(\Delta_2 + 1)} |V_1|$ such vertices. Therefore, since $t \leq
\Delta_2 + 1$, the set
$V_1'$ of remaining vertices has cardinality at least $|V_1|-t \cdot
\frac{1}{2(\Delta_2+1)}|V_1| \geq |V_1|/2$.

Since $(\alpha, \beta, \rho, \Delta_1, \Delta_2)$ is universal, the coloring
restricted to $V_1'$ is
$(\alpha,\beta,\rho,\Delta_1,\Delta_2)$-dense. Thus, the red graph is
$(\alpha,\beta,\rho,\Delta_1)$-dense (in which case we are done) or the blue
graph is $(\alpha,\beta,\rho,\Delta_2)$-dense. We may suppose the latter holds,
and there are subsets $T_1,\ldots,T_u$ each of cardinality at least $\alpha
| V_1'| \geq \frac{1}{2}\rho^{2\Delta_1}\alpha N$ and nonnegative integers
$e_1,\ldots,e_u$ such that $e_1+\cdots+e_u=\Delta_2-u+1$, and the following
holds:
\begin{itemize}
\item for $1 \leq i \leq u$, $T_i$ is bi-$(\rho^{2e_i},\rho)$-dense, and
\item for $1 \leq i < j \leq u$, each vertex in $T_i$ has at least
$(1-\beta)|T_j|$ neighbors in $T_j$.
\end{itemize}

Note that
$e_1+\cdots+e_u+d_1+\cdots+d_t=\Delta_2-u+1+\Delta_2-t+1=(2\Delta_2+1)-(u+t)+1$.
Moreover,
$\beta \geq 4(\Delta_2+1)\rho$, implying that for all $1 \leq i \leq u$ and all
$1 \leq j \leq t$ every vertex in $T_i$ has at least $(1 - \beta) |S_j|$
neighbors in $S_j$.
Therefore, the sequence $T_1,\ldots,T_u,S_1,\ldots,S_t$ implies that the blue
graph is
$(\frac{1}{2}\rho^{2\Delta_1}\alpha,\beta,\rho,2\Delta_2+1)$-dense, completing
the proof.
\end{proof}

By symmetry, the above lemma implies that if $\beta \geq 4(\Delta_1+1)\rho$ and
$(\alpha,\beta,\rho,\Delta_1,\Delta_2)$
is universal, then
$(\frac{1}{2}\rho^{2\Delta_2}\alpha,\beta,\rho,2\Delta_1+1,\Delta_2)$ is also
universal.

As already mentioned, if $\epsilon>1/2$, every graph $G$ is vacuously
bi-$(\epsilon,\rho)$-dense. As $\rho^{2 \cdot 0}=1>1/2$, setting $t=1$ and
$S_1=V(G)$, we have that every graph $G$ is $(\alpha,\beta,\rho,0)$-dense. This
shows that $(1,2\rho,\rho,0,0)$ is universal, which is the base case $h=0$ in
the induction proof of the next lemma.

\begin{lemma}\label{secondlemma}
Let $h$ be a nonnegative integer and $D:=2^h-1$. Then
$(2^{-2h}\rho^{6D-4h},2(D+1)\rho,\rho,D,D)$ is universal.
\end{lemma}

\begin{proof}
As mentioned above, the proof is by induction on $h$, and the base case $h=0$
is satisfied. Suppose it is satisfied for $h$, and we wish
to show it for $h+1$. Let $D=2^h-1$, $D'=2D+1=2^{h+1}-1$, and
$\beta=4(D+1)\rho=2(D'+1)\rho \geq 2(D+1)\rho$. Recall that, for $\beta \geq
\beta'$, if $(\alpha, \beta', \rho, \Delta_1, \Delta_2)$ is universal then so
is $(\alpha, \beta, \rho, \Delta_1, \Delta_2)$. Therefore, since $(2^{-2h}
\rho^{6D-4h}, 2(D+1)\rho, \rho, D, D)$ is universal, $(2^{-2h} \rho^{6D-4h},
\beta, \rho, D, D)$ is also. Applying Lemma \ref{firstlemma}, we have that
$(\frac{1}{2}\rho^{2D}2^{-2h}\rho^{6D-4h},\beta,\rho,D,2D+1)$ is universal.
Applying the symmetric version of Lemma \ref{firstlemma} mentioned above,
we have that
$$\Big(\frac{1}{2}\rho^{2(2D+1)}\frac{1}{2}\rho^{2D}2^{-2h}\rho^{6D-4h},\beta,\rho,2D+1,2D+1\Big)=(2^{-2(h+1)}\rho^{6D'-4(h+1)},\beta,\rho,D',D'),$$
is universal, which completes the proof by induction.
\end{proof}

We will use the following lemma of Lov\'asz \cite{L66}.

\begin{lemma}\label{basiclemma}
If $H$ has maximum degree $\Delta$ and $d_1,\ldots,d_t$ are nonnegative
integers satisfying
$d_1+\cdots+d_t=\Delta-t+1$, then there is a partition $V(H)=V_1 \cup \ldots
\cup V_t$ such that for $1 \leq i \leq t$, the induced subgraph
of $H$ with vertex set $V_i$ has maximum degree at most $d_i$.
\end{lemma}

The next simple lemma shows that in a bi-$(\epsilon,\rho)$-dense graph, for any
large vertex subset $B$, there are few vertices with few neighbors in $B$.

\begin{lemma}\label{previous}
If $G$ is a bi-$(\epsilon,\rho)$-dense graph on $n$ vertices with $\epsilon
\geq 1/n$ and $B \subset V(G)$ with $|B| \geq 2\epsilon n$, then there are less
than $3\epsilon n$ vertices in $G$ with fewer than $\frac{\rho}{2} |B|$
neighbors
in $B$.
\end{lemma}

\begin{proof}
Suppose for contradiction that the set $A$ of vertices in $G$ with fewer than
$\frac{\rho}{2}|B|$ neighbors in $B$ satisfies
$|A| \geq 3\epsilon n$. Partition $A \cap B=C_1 \cup C_2$ with $|C_1| \leq
| C_2|$ into two sets of size as equal as possible. Then the sets $A'=A
\setminus C_2$ and $B'= B \setminus C_1$ are disjoint, $|A'| \geq \lfloor |A|/2
\rfloor \geq \epsilon n$, $|B'| \geq |B|/2 \geq \epsilon n$, the number of
edges between $A'$ and $B'$ is less than $|A'|\frac{\rho}{2}|B|$, and the edge
density between $A'$ and $B'$ is less than
$\frac{|A'|\frac{\rho}{2}|B|}{|A'||B'|}=\frac{\rho}{2}\frac{|B|}{|B'|} \leq
\rho$, contradicting $G$ is bi-$(\epsilon,\rho)$-dense.
\end{proof}

The following embedding lemma is the last ingredient for the proof of Theorem
\ref{MaxDegreeIntro}.

\begin{lemma}\label{embedlemma}
If $\rho \leq 1/30$ and $G$ is a graph on $N \geq 4(2/\rho)^{2
\Delta}\alpha^{-1}
n$ vertices which is $(\alpha,\frac{1}{2\Delta},\rho,\Delta)$-dense, then $G$
contains every graph $H$ on $n$ vertices with maximum degree at most $\Delta$.
\end{lemma}
\begin{proof}
Since $G$ is $(\alpha,\frac{1}{2\Delta},\rho,\Delta)$-dense, there is a
sequence $S_1,\ldots,S_t$ of disjoint vertex subsets each of cardinality at
least $\alpha N$ and nonnegative integers $d_1,\ldots,d_t$ such that
$d_1+\cdots+d_t=\Delta-t+1$, and the following holds:
\begin{itemize}
\item for $1 \leq i \leq t$, $S_i$ is bi-$(\rho^{2d_i},\rho)$-dense, and
\item for $1 \leq i < j \leq t$, each vertex in $S_i$ has at least
$(1-\frac{1}{2\Delta})|S_j|$ neighbors in $S_j$.
\end{itemize}
By Lemma \ref{basiclemma}, there is a vertex partition $V(H)=V_1 \cup \ldots
\cup V_t$ such that the maximum degree of the induced subgraph of $H$ with
vertex set $V_i$ is at most $d_i$ for $1 \leq i \leq t$. Let $v_1,\ldots,v_n$
be an ordering of the vertices in $V(H)$ such that the vertices in $V_i$ come
before the vertices in $V_j$ for $i<j$. Let $N(h,k)$ denote the set of
neighbors $v_i$ of $v_k$ with $i \leq h$. For $v_k \in V_j$, let $M(h,k)$
denote the set of neighbors $v_i \in V_j$ of $v_k$ with $i \leq h$, that is,
$M(h,k) = N(h,k) \cap V_j$. Notice that
$|M(h,k)| \leq d_j$ for $v_k \in V_j$ since the induced subgraph of $H$ with
vertex set $V_j$ has maximum degree at most $d_j$.

We will find an embedding $f:V(H) \rightarrow V(G)$ of $H$ in $G$ such that
$f(V_i) \subset S_i$ for each $i$.  We will embed the vertices in increasing
order of their indices.
The embedding will have the property that after embedding the first $h$
vertices, if $k>h$ and $v_k \in V_j$, then the set $S(h,k)$ of vertices in
$S_j$ adjacent to all vertices in $f(N(h,k))$ has cardinality at least
$\frac{1}{2}(\rho/2)^{|M(h,k)|}|S_j|$. Notice that this condition is trivially
satisfied when $h=0$. Suppose that this condition is satisfied after embedding
the first $h$ vertices. The set $S(h,k)$ are the potential
vertices in which to embed $v_k$ after the first $h$ vertices have been
embedded, though
this set may already contain embedded vertices.

Let $j$ be such that $v_{h+1} \in V_{j}$. We need to find a vertex in
$S(h,h+1)$ to embed the copy of $v_{h+1}$. We have $$|S(h,h+1)| \geq
\frac{1}{2}(\rho/2)^{|M(h,h+1)|}|S_j| \geq \frac{1}{2}(\rho/2)^{d_j}|S_j|$$
since $|M(h,h+1)| \leq d_j$. If $d_j = 0$, we may pick $f(v_{h+1})$ to be any
element of the set $S(h,h+1) \char92 \{f(v_1), \dots, f(v_h)\}$. We may assume,
therefore, that $1 \leq d_j \leq \Delta$. In this case we know, for each of the
at most $d_j$ neighbors $v_k$ of $v_{h+1}$ with $k > h+1$ that are in $V_j$,
that the set $S(h,k)$ has cardinality at least
$\frac{1}{2}(\rho/2)^{d_j}|S_j|$. Let $\epsilon = \rho^{2 d_j}$.
Since, for $1 \leq d_j \leq \Delta$ and $\rho \leq 1/30$, $S_j$ is
bi-$(\rho^{2d_j},\rho)$-dense,
$|S(h,k)| \geq \frac{1}{2} (\rho/2)^{d_j} |S_j| \geq 2 \rho^{2 d_j} |S_j| = 2
\epsilon |S_j|$
and $\epsilon |S_j| = \rho^{2 d_j} |S_j| \geq \rho^{2 \Delta} \alpha N \geq
1$, we may apply Lemma \ref{previous} in $S_j$ with $B = S(h,k)$.
Therefore, for each vertex $v_k \in V_j, k>h+1$  adjacent to $v_{h+1}$, at most
$3 \rho^{2d_j}|S_j|$ vertices in $S_j$ have fewer than
$\frac{\rho}{2}|S(h,k)|$ neighbors in $S(h,k)$. Thus, all but at most $d_j
\cdot 3 \rho^{2d_j}|S_j|$ vertices
in $S_j$ have at least $\frac{\rho}{2}|S(h,k)|$ neighbors in $S(h,k)$ for all
$v_k \in V_j, k>h+1$  that are neighbors of $v_{h+1}$.
Since, for $\rho \leq 1/30$, we have $d_j \cdot 3 \rho^{2 d_j} \leq \frac{1}{4}
(\rho/2)^{d_j}$, there are at least
\begin{eqnarray*}
| S(h,h+1)|-d_j \cdot 3\rho^{2d_j}|S_j|-h &\geq&
\frac{1}{2}(\rho/2)^{d_j}|S_j|-d_j \cdot 3\rho^{2d_j}|S_j|-h \geq
\frac{1}{4}(\rho/2)^{d_j}|S_j|-h \\
&\geq&  \frac{1}{4}(\rho/2)^{\Delta}\alpha N-h \geq (2/\rho)^{\Delta}n-h > 0
\end{eqnarray*}
such vertices that are not
already embedded. We can pick any of these vertices to be $f(v_{h+1})$. To
continue, it remains to check that any such choice preserves the properties of
our embedding. Indeed,
\begin{itemize}
\item for any $k<h+1$ for which $v_{h+1}$ is adjacent to $v_k$, $f(v_{h+1})$ is
adjacent to $f(v_k)$;
\item if $k > h+1$ and $v_k$ and $v_{h+1}$ are not adjacent, then
$S(h+1,k)=S(h,k)$ and $M(h+1,k)=M(h,k)$;
\item if, for some $k > h+1$, $v_k$ and $v_{h+1}$ are adjacent and $v_k \in
V_{\ell}$ with $\ell \neq j$, then
$M(h+1,k)=0$ since vertices of $V_j$ are embedded before vertices of
$V_{\ell}$, $\ell>j$, so no vertex of $V_{\ell}$ was embedded yet.
Also, $|S(h+1,k)|\geq \frac{1}{2}|S_{\ell}|$ since
$|N(h+1,k)| \leq \Delta$, the vertices in $f(N(h+1,k))$ each have at least
$(1-\frac{1}{2\Delta})|S_{\ell}|$ neighbors in $S_{\ell}$, and hence
$|S(h+1,k)| \geq |S_{\ell}|-\Delta \cdot
\frac{1}{2\Delta}|S_{\ell}|=\frac{1}{2}|S_{\ell}|$;
\item if $k > h+1$, $v_k$ and $v_{h+1}$ are adjacent and $v_k
\in V_j$, then $|M(h+1,k)| = |M(h,k)| + 1$. Moreover, by our choice of the
vertex $f(v_{h+1})$, it has at least $\frac{\rho}{2}|S(h,k)|$ neighbors in
$S(h,k)$.
Therefore
$|S(h+1,k)|\geq \frac{\rho}{2}|S(h,k)| \geq \frac{1}{2} (\rho/2)^{|M(h,k)| + 1}
| S_j| = \frac{1}{2} (\rho/2)^{|M(h+1, k)|} |S_j|$, as required.
\end{itemize}

As we supposed there is an embedding of the first $h$ vertices with the desired
property, the above four facts imply
that there is an embedding of the first $h+1$ vertices with the desired
property. By induction on $h$, we find an embedding of $H$ in $G$.
\end{proof}

We can now prove the following theorem, which implies Theorem
\ref{MaxDegreeIntro}.

\begin{theorem}
For every $2$-edge-coloring of $K_N$ with $N=2^{84\Delta+2}\Delta^{32\Delta}n$,
at least one of the color classes contains a copy of every graph on $n$
vertices with
maximum degree $\Delta \geq 2$.
\end{theorem}
\begin{proof}
Let $h$ be the smallest positive integer such that $D:=2^h-1 \geq \Delta$. By
the definition of $D$, $\Delta \leq D<2\Delta$. Let $\rho=\frac{1}{8D^2}$,
$\alpha=2^{-2h}\rho^{6D-4h} \geq \rho^{6D}$, and $\beta=2(D+1)\rho \leq
\frac{1}{2D}$. Lemma \ref{secondlemma} implies that every red-blue coloring of
the edges of the complete graph $K_N$ is $(\alpha,\beta,\rho,D,D)$-dense. By
Lemma
\ref{embedlemma},
since
\begin{eqnarray*}
4(2/\rho)^{2D}\alpha^{-1} n & \leq & 4(16D^2)^{2D} \cdot (8D^2)^{6D} n \leq
4 (16 (2 \Delta)^2)^{4 \Delta} (8 (2\Delta)^2)^{12 \Delta} n\\
& = & 2^2 (2^6 \Delta^2)^{4 \Delta} (2^5 \Delta^2)^{12 \Delta} n = 2^{84 \Delta
+ 2} \Delta^{32 \Delta} n = N,
\end{eqnarray*}
at least one of the color classes contains a
copy of every graph on $n$ vertices with maximum degree $\Delta$.
\end{proof}

\section{Induced Ramsey numbers} \label{InducedSection}

The goal of this section is to prove Theorem \ref{maininduced}. We will do this
by
finding, in any $2$-edge-coloring of the pseudo-random graph $G$, a collection
of vertex subsets
$S_1, \ldots, S_t$ satisfying certain conditions. The conditions in question
are closely
related to the notion of density that we applied in the last section. Now, as
then, we demand
that the graph of one particular color satisfies a certain bi-density condition
within each $S_i$. In addition, we
demand that between the different $S_i$ the other color be sparse. This may
look like a
simple rearrangement of the condition from the previous section, but, given
that we are now
looking at colorings of a pseudo-random graph $G$ rather than the complete
graph $K_N$, the condition is
more general. Moreover, it is exactly what we need to make our embedding
lemma work.

\begin{definition}
An edge-coloring of a graph $G$ on $N$ vertices with colors
$1$ and $2$ is $(\alpha,\beta,\rho,f,\Delta_1,\Delta_2)$-dense if there is a
color $q \in \{1,2\}$, disjoint vertex subsets $S_1,\ldots,S_t$ each of
cardinality at least $\alpha N$ and nonnegative integers $d_1,\ldots,d_t$ with
$d_1+\cdots+d_t=\Delta_q-t+1$ such that the following holds:
\begin{itemize}
\item for $1 \leq i \leq t$, $S_i$ is bi-$(f(\rho,d_i),\rho)$-dense in the
graph
of color $q$, and
\item for $1 \leq i < j \leq t$, each vertex in $S_i$ is in at most
$\beta|S_j|$ edges of color $3-q$ with vertices in $S_j$.
\end{itemize}
\end{definition}

We say that $(\alpha,\beta,\rho,f,\Delta_1,\Delta_2)$ is {\it universal} if,
for every graph $G$,
every edge-coloring of $G$ with colors $1$ and $2$ is
$(\alpha,\beta,\rho,f,\Delta_1,\Delta_2)$-dense.
Note that the density condition used in the last section corresponds to the
case when $G=K_N$ and
$f(\rho,d_i)=\rho^{2d_i}$. Essentially the same proofs as Lemmas
\ref{firstlemma} and \ref{secondlemma} give the
following two more general lemmas. We include the proofs for completeness.

\begin{lemma}\label{firstlemmaind}
If $\beta \geq 4(\Delta_2+1)\rho$ and $(\alpha,\beta,\rho,f,\Delta_1,\Delta_2)$
is universal, then
$(\frac{1}{2}f(\rho,\Delta_1)\alpha,\beta,\rho,f,\Delta_1,2\Delta_2+1)$ is also
universal.
\end{lemma}
\begin{proof}
Consider an edge-coloring of a graph $G$ with colors $1$ and $2$. If the graph
of color $1$
is bi-$(f(\rho,\Delta_1),\rho)$-dense, then, taking, $q=1$,  $t=1$, $S_1=V(G)$
and
$d_1=\Delta_1$, we are done. So we may suppose that there are disjoint vertex
subsets $V_0,V_1$ with $|V_0|,|V_1| \geq
f(\rho,\Delta_1)N$ such that the density of color $1$ between them is less than
$\rho$.
Delete from $V_0$ all vertices in at least $2\rho|V_1|$ edges of color $1$ with
vertices
in $V_1$; the remaining subset $V_0'$ has cardinality at least
$\frac{1}{2}|V_0| \geq \frac{1}{2}f(\rho,\Delta_1)N$. Since $(\alpha, \beta,
\rho, f, \Delta_1, \Delta_2)$ is universal, the coloring restricted to the
induced subgraph of $G$ with vertex set $V_0'$ is
$(\alpha,\beta,\rho,f,\Delta_1,\Delta_2)$-dense. Thus, there is $q \in
\{1,2\}$, disjoint vertex subsets $S_1,\ldots,S_t \subset V_0'$ each of
cardinality at least $\alpha |V_0'|$ and nonnegative integers $d_1,\ldots,d_t$
with
$d_1+\cdots+d_t=\Delta_q-t+1$ such that the following holds:
\begin{itemize}
\item for $1 \leq i \leq t$, $S_i$ is bi-$(f(\rho,d_i),\rho)$-dense in the
graph
of color $q$, and
\item for $1 \leq i < j \leq t$, each vertex in $S_i$ is in at most
$\beta|S_j|$ edges of color $3-q$ with vertices in $S_j$.
\end{itemize}
If $q=1$, we are done. Therefore, we may suppose $q=2$.

Since each vertex in $V_0'$ (and hence in each $S_i$) is in at most
$2\rho|V_1|$ edges of color $1$ with vertices in $V_1$, then there are
at most $2\rho|S_i||V_1|$ edges of color $1$ between $S_i$ and $V_1$. For
$1\leq i \leq
t$, delete from $V_1$ all vertices in at least $4(\Delta_2+1)\rho|S_i|$ edges
of color $1$ with vertices in $S_i$. For any
given $i$, there can be at most $\frac{1}{2(\Delta_2 + 1)} |V_1|$ such
vertices. Therefore, since $t \leq
\Delta_2 + 1$, the set $V_1'$ of remaining vertices has cardinality at least
$|V_1|-t \cdot
\frac{1}{2(\Delta_2+1)}|V_1| \geq |V_1|/2$.

Since $(\alpha, \beta, \rho, f, \Delta_1, \Delta_2)$ is universal, the coloring
restricted to the induced subgraph of $G$ with vertex set $V_1'$ is
$(\alpha,\beta,\rho,f,\Delta_1,\Delta_2)$-dense. Thus, there is $q' \in
\{1,2\}$, disjoint vertex subsets $T_1,\ldots,T_u \subset V_1'$ each of
cardinality at least $\alpha |V_1'|$ and nonnegative integers $e_1,\ldots,e_u$
with
$e_1+\cdots+e_u=\Delta_{q'}-u+1$ such that the following holds:
\begin{itemize}
\item for $1 \leq i \leq u$, $T_i$ is bi-$(f(\rho,e_i),\rho)$-dense in the
graph
of color $q'$, and
\item for $1 \leq i < j \leq u$, each vertex in $T_i$ is in at most
$\beta|T_j|$ edges of color $3-q'$ with vertices in $T_j$.
\end{itemize}
If $q'=1$, we are done. Therefore, we may suppose $q'=2$.

Note that
$e_1+\cdots+e_u+d_1+\cdots+d_t=\Delta_2-u+1+\Delta_2-t+1=(2\Delta_2+1)-(u+t)+1$.
Moreover,
$\beta \geq 4(\Delta_2+1)\rho$, implying that for all $1 \leq i \leq u$ and all
$1 \leq j \leq t$ every vertex in $T_i$ is in at most $\beta|S_j|$
edges of color $1$ with vertices in $S_j$. Therefore, the sequence
$T_1,\ldots,T_u,S_1,\ldots,S_t$ implies that the edge-coloring of $G$
is
$(\frac{1}{2}f(\rho,\Delta_1)\alpha,\beta,\rho,f,\Delta_1,2\Delta_2+1)$-dense,
completing
the proof.
\end{proof}

By symmetry, the above lemma implies that if $\beta \geq 4(\Delta_1+1)\rho$ and
$(\alpha,\beta,\rho,f,\Delta_1,\Delta_2)$
is universal, then
$(\frac{1}{2}f(\rho,\Delta_2)\alpha,\beta,\rho,f,2\Delta_1+1,\Delta_2)$ is
also
universal.

\begin{lemma}\label{generallemma} Let $h$ be a nonnegative integer and $f$ be
such that $f(\rho,0)=1$.
Define $$\alpha_h=2^{-2h}f(\rho,0)^{-1}f(\rho,2^{h}-1)^{-1}\prod_{i=0}^h
f(\rho,2^i-1)^2.$$ Then $(\alpha_h,2^{h+1}\rho,\rho,f,2^{h}-1,2^{h}-1)$ is
universal.
\end{lemma}
\begin{proof}
The proof is by induction on $h$. As already mentioned, if $\epsilon>1/2$,
every graph $G$ is vacuously bi-$(\epsilon,\rho)$-dense. Since
$\alpha_0=1>1/2$, setting $t=1$ and
$S_1=V(G)$, we have $(1,2\rho,\rho,f,0,0)$ is universal, which is the base case
$h=0$.

Suppose the lemma is satisfied for $h$, and we wish to show it for $h+1$. Let
$D=2^h-1$, $D'=2D+1=2^{h+1}-1$, and
$\beta=4(D+1)\rho=2(D'+1)\rho=2^{h+2}\rho$. Note that, for $\beta \geq
\beta'$, if $(\alpha, \beta', \rho, f, \Delta_1, \Delta_2)$ is universal then
so
is $(\alpha, \beta, \rho, f, \Delta_1, \Delta_2)$. Therefore, since $(\alpha_h,
2(D+1)\rho, \rho,f, D, D)$ is universal, $(\alpha_h,
\beta, \rho, f, D, D)$ is also. Applying Lemma \ref{firstlemmaind}, we have
that
$(\frac{1}{2}f(\rho,D)\alpha_h,\beta,\rho,f,D,2D+1)$ is universal.
Applying the symmetric version of Lemma \ref{firstlemmaind} mentioned above,
we have that
$$\Big(\frac{1}{2}f(\rho,2D+1)\frac{1}{2}f(\rho,D)\alpha_h,\beta,\rho,f,2D+1,2D+1\Big)=(\alpha_{h+1},\beta,\rho,f,D',D'),$$
is universal, which completes the proof by induction.
\end{proof}

A graph $G$ is {\it $n$-Ramsey-universal} if, in any $2$-edge-coloring of $G$,
there are monochromatic induced copies of every graph on $n$ vertices all of
the same color. The following lemma implies Theorem \ref{maininduced}.

\begin{lemma}\label{inducedembeddinglemma} If $G$ is
$(1/2,\lambda)$-pseudo-random on $N$ vertices with $\lambda \leq
2^{-140 n} n^{-40n}N$, then $G$ is $n$-Ramsey-universal.
\end{lemma}

The set-up for the proof of this lemma is roughly similar to the one presented
in the previous section. We start with a collection of bi-dense sets, in say
blue, such that the density of red edges between each pair of sets is
small. The goal is to embed a blue induced copy of a given graph $H$ on
vertices $1,\ldots,n$. We embed vertices one at a time, always maintaining
large sets in which we may embed later vertices. Suppose that at
step $i$ of our embedding, after $v_1, v_2, \dots, v_i$ are chosen, we have
sets $V_{j,i}$ for $j>i$ corresponding to the possible choices for future
$v_j$. If the vertices $j,\ell>i$ are not adjacent, then, by the
pseudo-randomness of $G$, the density of nonedges between any two large sets is
roughly $1/2$, and it is therefore easy to guarantee that we can pick $v_j$ and
$v_{\ell}$ so that they are nonadjacent. On the other hand, if the vertices
$j,\ell>i$ are adjacent, then we need to guarantee that $v_j$ and $v_{\ell}$
will be joined by a blue edge. Thus, it would be helpful to ensure that the
density of blue edges between $V_{j,i}$ and $V_{\ell,i}$ is not too small. In
the bounded-degree case we maintain such a property by
exploiting the fact that the blue density between any two large sets is large.
Here, we do not have this luxury in the case that $V_{j,i}$ and $V_{\ell,i}$
are subsets of different bi-dense sets in the collection. It is instead
necessary to use the fact that the underlying graph $G$ is pseudo-random.

To see how this helps, suppose that we now wish to embed $v_{i+1}$. This will
affect the sets $V_{j,i}$ and $V_{\ell,i}$, resulting in
subsets $V_{j,i+1}$ and $V_{\ell,i+1}$. We would like these subsets to mirror
the density
properties between $V_{j,i}$ and $V_{\ell,i}$. The way we proceed is to show
that using pseudo-randomness we can choose $v_{i+1}$ such that the density of
red edges between the sets $V_{j,i+1}$ and $V_{\ell,i+1}$ remains small. Since
$G$ is pseudo-random, the total density between large sets is roughly
$1/2$ and therefore there will still be many blue edges between these two sets.

{\bf Proof of Lemma \ref{inducedembeddinglemma}:}
We split the proof into four steps.

{\bf Step 1:} We will first choose appropriate constants and prepare $G$ for
embedding monochromatic induced subgraphs.

Any $(1/2,\lambda)$-pseudo-random graph on at least two vertices must satisfy
$\lambda \geq 1/2$. Indeed, letting $A$ and $B$ be distinct vertex subsets each
of cardinality $1$, we have $$1/2=|d(A,B)-1/2| \leq
\frac{\lambda}{\sqrt{|A||B|}} =\lambda.$$ It follows that $N \geq
2^{140n}n^{40n}\lambda \geq 2^{138n}n^{40n}$.

We will start by picking some constants. Pick $\rho = 2^{-13}n^{-3}$, $h=\lceil
\log n \rceil \leq \log 2n$,
$\beta=2^{h+1}\rho \leq 8n\rho=2^{-10}n^{-2}$, $f(\rho,0)=1$ and
$f(\rho,d)=2^{-5n}\rho^d$ if $d>0$, so
\begin{eqnarray*}
\alpha & = & 2^{-2h} f(\rho,0)^{-1}f(\rho,2^{h}-1)^{-1}\prod_{i=0}^{h}
f(\rho,2^i-1)^2 = 2^{-2h} f(\rho,2^{h}-1)\prod_{i=1}^{h-1} f(\rho,2^i-1)^2 \\ &
= &
2^{-2h -5n} \rho^{2^h - 1} \prod_{i=1}^{h-1} 2^{-10n} \rho^{2(2^i - 1)}
 =  2^{-2h-(2h-1)5n}\rho^{3 \cdot 2^h-2h-3} \\ & \geq & (2n)^{-12n}\rho^{6n} =
2^{-90n}n^{-30n}.
\end{eqnarray*}
Lemma \ref{generallemma} implies that $(\alpha,\beta,\rho,f,2^{h}-1,2^h-1)$ is
universal. As $2^h \geq n$, it follows that $(\alpha,\beta,\rho,f,n-1,n-1)$ is
also universal. Let $\epsilon_1=\frac{1}{2n}$,
$\epsilon_2=\frac{\epsilon_1\rho}{32n} = 2^{-19}n^{-5}$,
$\epsilon_3=\epsilon_4=\frac{1}{8n}$, $\epsilon_5=\frac{1}{8n^2}$,
$\epsilon_6=\epsilon_2\epsilon_5 = 2^{-22}n^{-7}$ and
$\beta'=2n\beta \leq 2^{-9}n^{-1}$.

Since every red-blue edge of $G$ is $(\alpha,\beta,\rho, f,n-1,n-1)$-dense, we
may assume that there are disjoint vertex subsets $S_1,\ldots,S_t$ each of
cardinality at least $\alpha N$ and nonnegative integers $d_1,\ldots,d_t$ with
$d_1+\cdots+d_t=n-t$ such that
\begin{itemize}
\item for $1 \leq i \leq t$, $S_i$ is bi-$(f(\rho,d_i),\rho)$-dense in the blue
graph, and
\item for $1 \leq i < j \leq t$, each vertex in $S_i$ is in at most
$\beta|S_j|$ red edges with vertices in $S_j$.
\end{itemize}

We will show that we can find a monochromatic blue induced copy of each graph
$H$ on $n$ vertices. We may suppose the vertex set of $H$ is
$V(H)=[n]:=\{1,\ldots,n\}$. Partition $[n]=U_1 \cup \ldots \cup U_t$, with the
vertices in $U_i$ coming before the vertices in $U_j$ for $i<j$ and
$|U_i|=d_i+1$ for $1 \leq i \leq t$. For $j \in U_l$, let $D(i,j)$ denote the
number of neighbors $h$ of $j$ with
$h \leq i$ and $h \in U_l$. Arbitrarily partition $S_i$ into $d_i+1$
sets $\bigcup_{k \in U_i} V_k$ each of cardinality at least $\lfloor
\frac{|S_i|}{d_i+1}\rfloor \geq \lfloor
\frac{|S_i|}{n}\rfloor \geq \frac{|S_i|}{2n} \geq \frac{\alpha}{2n} N$, where
we use $d_i+1 \leq n$, $|S_i| \geq \alpha N$, and the lower bounds on $\alpha$
and $N$.

{\bf Step 2:} We now describe our strategy for constructing induced blue copies
of $H$. In broad outline, we
proceed by induction, embedding each successive vertex $i$ in the set $V_i$.
To achieve this, we have to maintain several conditions which allow us to embed
future vertices.

At the end of step $i$, we will have vertices $v_1,\ldots,v_i$ and, for $j >
i$, subsets
$V_{j,i} \subset V_j$ such that the following four conditions hold.
\begin{enumerate}
\item  for $j, \ell \leq i$, if $(j,{\ell})$ is an
edge of $H$, then $(v_j,v_{\ell})$ is a blue edge of $G$, otherwise
$v_j$ and $v_{\ell}$  are not adjacent in $G$;
\item for $ j \leq i < \ell$, if $(j,{\ell})$ is an
edge of $H$, then $v_j$ is adjacent to all vertices in $V_{\ell,i}$
by blue edges, otherwise there are no edges of $G$ from $v_j$ to
$V_{\ell,i}$;
\item for $j > i$, we have $|V_{j,i}| \geq 4^{-i}\rho^{D(i,j)}|V_j|$;
\item for $\ell>j>i$, if $j \in U_{q_1}$ and $\ell \in U_{q_2}$ with
$q_1<q_2$, then each vertex in $V_{j,i}$ is in at most $(1+\epsilon_1)^i
\beta'|V_{\ell,i}|$ red edges with vertices in $V_{\ell,i}$.
\end{enumerate}
Note that $V_{j,i}$ is a subset of vertices of $G$ in which we can still embed
vertex $j$ from $H$ after $i$ steps of our embedding procedure.
Clearly, at the end of the first $n$ steps of this process we obtain the
required copy of $H$. For $i=0$ and $j \in [n]$, define
$V_{j,0}=V_j$. Notice that the above four properties are satisfied
for $i=0$. Indeed, the first two properties are vacuously satisfied, the third
property follows from $V_{j,0}=V_j$, and the last property follows from the
simple inequality $\beta'|V_{\ell,0}|
= 2n\beta|V_{\ell}|\geq \beta|S_{q_2}|$. We now assume that the above four
properties are satisfied at the end
of step $i$ and show how to complete step $i+1$ by finding a vertex
$v_{i+1} \in V_{i+1,i}$ and, for $j>i+1$, subsets $V_{j,i+1} \subset V_{j,i}$
such that conditions 1-4 still hold.

Before we begin the next step of the proof, we need to introduce some notation.
For a vertex $w \in V_j$ and a
subset $S \subset V_{\ell}$ with $j \not =\ell$, let
\begin{itemize}
\item $N(w,S)$ denote the set of vertices $s \in S$ such that $(s,w)$ is an
edge of
$G$,
\item $R(w,S)$ denote the set of vertices $s \in S$ such that $(s,w)$ is a
red edge of $G$,
\item $B(w,S)$ denote the set of vertices $s \in S$ such that $(s,w)$ is a
blue edge of $G$,
\item $\tilde{N}(w,S)=N(w,S)$ if $(j, \ell)$ is an edge of $H$ and
$\tilde{N}(w,S)=S
\setminus N(w,S)$ otherwise,
\item $\tilde{B}(w,S)=B(w,S)$ if $(j, \ell)$ is an edge of $H$ and
$\tilde{B}(w,S):=S
\setminus N(w,S)$ otherwise.
\end{itemize}

Note, for all $S \subset V_{\ell}$ and $w \in V_j$, that
$\tilde{B}(w,S)=\tilde{N}(w,S) \setminus R(w,S)$.
Moreover, since the graph $G$ is pseudo-random with
edge density $1/2$, we expect that for every large subset $S \subset V_{\ell}$
and for most vertices $w \in V_j$ the size of $\tilde{N}(w,S)$ will be roughly
$|S|/2$.

{\bf Step 3:} We next show that if there is a vertex satisfying certain
conditions, then we can continue our embedding. In the last step we show that
there is such a ``good'' vertex.

Let $q$ be the index such that $i+1 \in U_q$. Call a vertex $w \in V_{i+1,i}$
{\it good}
if
\begin{enumerate}
\item for all $j >i+1$ such that $(j,i+1)$ is an edge of $H$ and $j \in U_q$,
$|B(w,V_{j,i})|\geq \rho|V_{j,i}|$,
\item for all $j >i+1$, $|\tilde{N}(w,V_{j,i})|\geq
\left(\frac{1}{2}-\frac{\epsilon_1}{20}\right)|V_{j,i}|$,
\item for all $\ell>j>i+1$ with $j \in U_{q_1}$, $\ell \in U_{q_2}$, and $q_1 <
q_2$,
there are at most $\epsilon_2|V_{j,i}|$ vertices $y \in V_{j,i}$ such that $y$
is in at least
$\beta'\left(\frac{1}{2}-\frac{\epsilon_1}{10}\right)|V_{\ell,i}|$ red edges
with vertices in $V_{\ell,i}$ and $y$ is in at least
$\left(\frac{1}{2}+\frac{\epsilon_1}{10}\right)|R(y,V_{\ell,i})|$ red edges
with vertices of
$\tilde{N}(w,V_{\ell,i})$.
\end{enumerate}
Note that, because the graph $G$ is pseudo-random with edge density $1/2$, we
expect a typical vertex in $V_{i+1,i}$ to be adjacent (and also nonadjacent) to
roughly $1/2$ of the vertices in $V_{j,i}$ and $V_{\ell,i}$. Moreover,
condition $3$ roughly says that for a typical vertex, the density of red edges
between its  neighborhoods in $V_{j,i}$ and $V_{\ell,i}$ is not much larger
than the overall density of red edges between these two sets.

We will now show that if there is a good vertex $w \in V_{i+1,i}$, then we may
continue the
embedding by taking $v_{i+1}=w$ and, for $j > i + 1$ with $j \in U_{q_1}$,
letting
$V_{j,i+1}$ be the subset of $\tilde B(w,V_{j,i})$ formed by deleting all
vertices $y$ for which there is $\ell>j$ with $\ell \not \in U_{q_1}$ such that
$y$ is in at least
$\beta'\left(\frac{1}{2}-\frac{\epsilon_1}{10}\right)|V_{\ell,i}|$ red edges
with vertices in $V_{\ell,i}$ and $y$ is in at least
$\left(\frac{1}{2}+\frac{\epsilon_1}{10}\right)|R(y,V_{\ell,i})|$ red edges
with vertices of $\tilde{N}(w,V_{\ell,i})$. Note that, by the third property of
good vertices,
\begin{equation}\label{eq1234}|V_{j,i+1}| \geq |\tilde
B(w,V_{j,i})|-n\epsilon_2|V_{j,i}|.\end{equation}
Let us verify each of the required properties of our embedding in turn.

To verify the first property, we need to show that if $j \leq i$ and $(j, i+1)$
is
an edge of $H$ then $(v_j, v_{i+1})$ is a blue edge and, if $(j,i+1)$ is not an
edge
of $H$, then $(v_j, v_{i+1})$ is not in $G$. But this follows by induction
since, when the first
$i$ vertices were embedded, we had that for all $j \leq i < l$, if $(j,l)$ was
an edge of $H$, then $v_j$
was adjacent to all edges of $V_{l,i}$ by blue edges. Otherwise, there were no
edges between $v_j$ and $V_{l,i}$.
Taking $l = i+1$, the necessary property follows.

For the second property, we would like to show that for $j \leq i+1 < l$, if
$(j,l)$ is an edge of $H$, then $v_j$ is
adjacent to all vertices in $V_{l,i+1}$ by blue edges and, otherwise, there are
no edges between $v_j$ and $V_{l,i+1}$.
Observe that, for all $l > i+1$, the set $V_{l,i+1}$ is a subset of the set
$V_{l,i}$. Therefore, by induction,
we only need to check the condition for $j = i+1$. But $V_{l,i+1}$ is a subset
of $\tilde{B}(v_{i+1}, V_{l,i})$, so this
follows by definition.

We now wish to prove that, for all $j > i+1$, $|V_{j,i+1}| \geq 4^{-(i+1)}
\rho^{D(i+1,j)} |V_j|$. Inequality (\ref{eq1234}) together with the first
property of good vertices
implies that if $j>i+1$, $(j,i+1)$ is an edge of $H$ and $j \in U_q$ (recall
that also $i+1 \in U_q$), then,
since $\epsilon_2 \leq \rho/(2n)$ and $D(i+1,j) = D(i,j) + 1$,
$$|V_{j,i+1}| \geq (\rho-n\epsilon_2)|V_{j,i}| \geq \frac{\rho}{2}|V_{j,i}|
\geq
\frac{\rho}{2}4^{-i}\rho^{D(i,j)}|V_j| \geq 4^{-(i+1)}\rho^{D(i+1,j)}|V_j|.$$
Inequality (\ref{eq1234}), the second property of good vertices and the
inductive assumption that $w$ has at most $(1+\epsilon_1)^i \beta' |V_{j,i}|$
red neighbors in $V_{j,i}$ if $j \not\in U_q$
together imply that for all other $j>i+1$, we have
\begin{eqnarray*} |V_{j,i+1}| & \geq &
| \tilde{B}(w,V_{j,i})|-n\epsilon_2|V_{j,i}| = |\tilde{N}(w,V_{j,i}) \setminus
R(w,V_{j,i})| -n\epsilon_2|V_{j,i}| \\ & \geq &
\left(\frac{1}{2}-\frac{\epsilon_1}{20}\right)|V_{j,i}|-(1+\epsilon_1)^i\beta'|V_{j,i}|
-n\epsilon_2|V_{j,i}|
\geq
\left(\frac{1}{2}-\frac{\epsilon_1}{20}-3\beta'-n\epsilon_2\right)|V_{j,i}|
\\ & \geq &  \left(\frac{1}{2}-\frac{\epsilon_1}{10}\right)|V_{j,i}|
\geq \frac{1}{4}4^{-i}\rho^{D(i,j)}|V_j| = 4^{-(i+1)}\rho^{D(i+1,j)}|V_j|.
\end{eqnarray*}
Here we use that $\epsilon_1 = 1/2n$,
$\beta' \leq 2^{-9} n^{-1}$, $\epsilon_2 \leq \epsilon_1/32 n$, and $D(i+1,j)=
D(i,j)$ (since $i+1$ and $j$ are either nonadjacent or belong to different
$U$s).
In either case, the required lower bound on the cardinality of $V_{j,i+1}$
holds. Note the intermediate inequality that $|V_{l,i+1}| \geq
\left(\frac{1}{2} -
\frac{\epsilon_1}{10}\right) |V_{l,i}|$ whenever $l \not\in U_q$.

If $i+1<j<\ell$ is such that $j \in U_{q_1}$ and $\ell \in U_{q_2}$ with
$q \leq q_1<q_2$, our deletion of vertices from $\tilde{B}(w,V_{j,i})$ implies
that each vertex in  $V_{j,i+1}$ is in less than
$$\beta'\left(\frac{1}{2}-\frac{\epsilon_1}{10}\right)|V_{\ell,i}| \leq \beta'
| V_{\ell,i+1}|$$ red edges with vertices in $V_{\ell,i}$
or each vertex in  $V_{j,i+1}$ is in less than
\begin{eqnarray*}
\left(\frac{1}{2}+\frac{\epsilon_1}{10}\right)|R(y,V_{\ell,i})| & \leq &
\left(\frac{1}{2}+\frac{\epsilon_1}{10}\right)(1+\epsilon_1)^i\beta'|V_{\ell,i}|
\leq
\left(\frac{1}{2}+\frac{\epsilon_1}{10}\right)(1+\epsilon_1)^i\beta'|V_{\ell,i+1}|/\left(\frac{1}{2}-\frac{\epsilon_1}{10}\right)
\\ & \leq &
(1+\epsilon_1)^{i+1}\beta'|V_{\ell,i+1}|
\end{eqnarray*}
red edges with vertices of $\tilde{N}(w,V_{\ell,i})$. In either case, we see
that the last desired condition of the embedding is satisfied.

{\bf Step 4:} We have shown that if there is a good vertex, then we can
continue the embedding. In this step we show that there is a good vertex in
$V_{i+1,i}$, which completes the proof.

The next three claims imply that the fraction of vertices in $V_{i+1,i}$ that
are
good is at least $1-n\epsilon_3-n\epsilon_4-n^2\epsilon_5>1/2$, i.e., more than
half of the vertices of $V_{i+1,i}$ are good. Indeed, Claim \ref{claim1} shows
that the first property of good vertices is satisfied for all but at most
$n\epsilon_3|V_{i+1,i}|$
vertices in $V_{i+1,i}$. Claim \ref{claim2} shows that the second property of
good
vertices is satisfied for all but at most $n\epsilon_4|V_{i+1,i}|$ vertices in
$V_{i+1,i}$. Claim 3 shows that the third property of good vertices is
satisfied for all but at most $n^2\epsilon_5|V_{i+1,i}|$ of the vertices in
$V_{i+1,i}$. These three claims therefore complete the proof. \qed

\begin{claim}\label{claim1}
For $j >i+1$ such that $(j,i+1)$ is an edge of $H$ and $j \in U_q$, let $Q_j$
denote the set of vertices $w \in V_{i+1,i}$ such that
$|B(w,V_{j,i})|< \rho|V_{j,i}|$. Then $|Q_j| < \epsilon_3|V_{i+1,i}|$.
\end{claim}
\begin{proof}
Suppose, for contradiction, that $|Q_j| \geq \epsilon_3 |V_{i+1,i}|$. As $j,i+1
\in U_q$ and $|U_q|=d_q+1$, we have $d_q \geq 1$ and
$f(\rho,d_q)=2^{-5n}\rho^{d_q}$. Since $2^{3n} \geq 8n^2$, $|V_{i+1, i}|
\geq 4^{-i} \rho^{D(i, i+1)} |V_{i+1}|$ and $|V_{i+1}| \geq |S_q|/2n$, we have
$$|Q_j| \geq
\epsilon_3 4^{-i}\rho^{D(i,i+1)}|V_{i+1}| \geq \epsilon_3
4^{1-n}\rho^{d_q}|V_{i+1}| \geq
\frac{\epsilon_3}{n}4^{-n}\rho^{d_q}|S_q| \geq f(\rho,d_q)|S_q|.$$
We also have $$|V_{j,i}| \geq 4^{-i}\rho^{D(i,j)}|V_j| \geq
4^{1-n}\rho^{d_q}|V_j| \geq n^{-1}4^{-n}\rho^{d_q}|S_q| \geq
f(\rho,d_q)|S_q|.$$
Since $S_q$ is bi-$(f(\rho,d_q),\rho)$-dense in blue, the blue edge
density between $Q_j$ and $V_{j,i}$ is at least $\rho$, contradicting the
definition of $Q_j$.
\end{proof}

\begin{claim}\label{claim2}
For $j >i+1$, let $P_j$ denote the set of vertices $w \in V_{i+1,i}$ such that
$$|\tilde{N}(w,V_{j,i})|<
\left(\frac{1}{2}-\frac{\epsilon_1}{20}\right)|V_{j,i}|.$$ Then $|P_j| <
\epsilon_4|V_{i+1,i}|$.
\end{claim}
\begin{proof}
The definition of $P_j$ implies that the density of edges between $P_j$ and
$V_{j,i}$ is either less than $\frac{1}{2} - \frac{\epsilon_1}{20}$ or more
than
$\frac{1}{2}+\frac{\epsilon_1}{20}$ (depending on whether or not $(i+1,j)$ is
an edge of $H$). Therefore,
since $G$ is $(1/2,\lambda)$-pseudo-random, we have
$\frac{\epsilon_1}{20}<\frac{\lambda}{\sqrt{|P_j||V_{j,i}|}}$.  Note that, for
$j > i$, since $\rho = 2^{-13} n^{-3}$ and $\alpha \geq 2^{-90n} n^{-30 n}$,
\begin{equation}\label{eqn2} |V_{j,i}| \geq 4^{-i} \rho^{D(i,j)} |V_j| \geq
4^{-n} \rho^{n} |V_j| \geq 2^{-15n} n^{-3n} \frac{\alpha}{2n}
N \geq 2^{-106n} n^{-34n} N.\end{equation}
Hence, since we also have $\epsilon_1 = 1/2n$, $\epsilon_4 = 1/8n$ and $\lambda
\leq 2^{-140n} n^{-40n} N$,
\begin{eqnarray}\label{eqn3}
| P_j| & < & \frac{400\lambda^2}{\epsilon_1^2 |V_{j,i}|} < \frac{2^{9}
| 2^{-280n}
n^{-80n} N^2}{(2n)^{-2} 2^{-106n} n^{-34n} N} = 2^{11} n^2 2^{-174 n}
n^{-46n} N\\
& \leq & 2^{-163n} n^{-44n} N \leq (8n)^{-1} 2^{-106n} n^{-34n}N \leq
\epsilon_4|V_{i+1,i}|.\nonumber
\end{eqnarray}
\end{proof}

\begin{claim}\label{claim3}
Fix a pair $j$ and $\ell$ with $i+1<j<\ell$, $j \in U_{q_1}$, $\ell \in
U_{q_2}$, and $q_1 < q_2$. Let $X=V_{i+1,i}$, $Y=V_{j,i}$, and $Z=V_{\ell,i}$.
Define the bipartite graph $F=F_{j,\ell}$ with parts $X$ and $Y$ where $(x,y)
\in X \times Y$ is an edge if
$$|R(y,Z)| \geq \beta' \left(\frac{1}{2}-\frac{\epsilon_1}{10}\right)|Z|$$ and
$$|R(y,Z) \cap \tilde{N}(x,Z)| >
\left(\frac{1}{2}+\frac{\epsilon_1}{10}\right)|R(y,Z)|.$$ Let $T_{j,\ell}$
denote the set of vertices in
$X$ with degree at least $\epsilon_2|Y|$ in $F$. Then $|T_{j,\ell}| \leq
\epsilon_5|X|$.
\end{claim}
\begin{proof}
For $y \in Y$, let $X_y \subset X$ denote the neighbors of $y$ in graph $F$.
Note that, for
every $x \in X_y$, the fact that $|R(y,Z) \cap \tilde{N}(x,Z)| >
\left(\frac{1}{2}+\frac{\epsilon_1}{10}\right)|R(y,Z)|$ implies that, in either
the graph $G$ or its complement,
$x$ has at least $\left(\frac{1}{2}+\frac{\epsilon_1}{10}\right)|R(y,Z)|$
neighbors in $R(y,Z)$ (this is again because $\tilde{N}(x,Z)$ is either the
neighborhood of $x$ or its complement depending on whether or not $(i+1,\ell)$
is an edge of $H$).
Therefore, since $G$ is $(1/2,\lambda)$-pseudo-random, $$\frac{\epsilon_1}{10}
\leq
\frac{\lambda}{\sqrt{|X_y||R(y,Z)|}},$$
Note that, by the first condition on $F$, if $y$ has any
neighbors in $X$, $|R(y,Z)| \geq \beta' |Z|/4$. Therefore,
$$|X_y| \leq \frac{100\lambda^2}{\epsilon_1^2|R(y,Z)|} \leq
\frac{400\lambda^2}{\epsilon_1^2\beta'|Z|} \leq \epsilon_6|X|.$$
This last inequality follows as in the previous claim. Indeed, since
$\beta' = 2^{-9} n^{-1}$, $\epsilon_6 = 2^{-22} n^{-7}$, $Z=V_{\ell,i}$, and
$X=V_{i+1,i}$, using inequalities (\ref{eqn2},\ref{eqn3}), we have
$$\frac{400 \lambda^2}{\epsilon_1^2 \beta' |Z|} \leq \beta'^{-1} 2^{-163n}
n^{-44n} N \leq 2^{-154 n} n^{-43 n} N \leq 2^{-22} n^{-7} 2^{-106n} n^{-34n} N
\leq \epsilon_6 |X|.$$
Therefore, the edge density of $F$ between $X$ and $Y$ is at most $\epsilon_6$
and there are at most
$\frac{\epsilon_6|X||Y|}{\epsilon_2|Y|}=\epsilon_5|X|$ vertices in $X$ with
degree at least $\epsilon_2|Y|$ in $F$.
\end{proof}

\section{Concluding remarks}

Another interesting concept of sparseness, introduced by Chen and Schelp
\cite{CS93}, is that of arrangeability. A graph $H$ is said to be {\it
$p$-arrangeable} if there is an ordering of the vertices of $H$ such that, for
any vertex $v_i$, the set of neighbors to the right of $v_i$ in the ordering
have at most $p$ neighbors to the left of $v_i$ (including $v_i$ itself).
Extending the result of Chv\'atal, R\"{o}dl, Szemer\'edi and Trotter
\cite{CRST83}, Chen and Schelp showed that for every $p$ there is a constant
$c(p)$ such that, for any $p$-arrangeable graph $H$ with $n$ vertices, $r(H)
\leq c(p) n$. This result has several consequences. Planar graphs, for example,
may be shown to be $10$-arrangeable \cite{KT94}, so their Ramsey numbers grow
linearly. The best bound that is known for $c(p)$, again due to Graham,
R\"{o}dl and Ruci\'nski \cite{GRR00}, is $c(p) \leq 2^{c p (\log p)^2}$.
Unfortunately, it
is unclear whether the bounds that we have given for bounded-degree graphs can
be extended to the class of arrangeable graphs. It would be interesting to
prove such a bound.

An even more problematic notion is that of degeneracy. A graph $H$ is said to
be {\it $d$-degenerate} if there is an ordering of the vertices of $H$ such
that any vertex $v_i$ has at most $d$ neighbors that precede it in the
ordering. Equivalently, every subgraph of $H$ has a vertex of degree at most
$d$. A conjecture of Burr and Erd\H{o}s \cite{BE75} states
that for every $d$ there should be a constant $c(d)$ such that, for any
$d$-degenerate graph $H$ with $n$ vertices, $r(H) \leq c(d) n$. This
conjecture,
which is still open, is a substantial generalization of the results on Ramsey
numbers of bounded-degree graphs. The best result that is known, due to Fox and
Sudakov \cite{FS209}, is $r(H) \leq 2^{c(d) \sqrt{\log n}} n$.

An old related problem is to bound the Ramsey number of graphs with $m$ edges.
Erd\H{o}s and Graham \cite{ErGr} conjectured that among all graphs with
$m= \binom{n}{2}$ edges and no isolated vertices, the complete graph on
$n$ vertices has the largest Ramsey number. Motivated by the lack of
progress on this conjecture, Erd\H{o}s \cite{Er1} asked whether one could
at least show that the Ramsey number of any graph with $m$ edges is not
much larger than that of the complete graph with the same size. Since the
number of vertices in a complete graph with $m$ edges is on the order of
$\sqrt{m}$, Erd\H{o}s conjectured that $r(H) \leq 2^{c\sqrt{m}}$ holds for
every graph $H$ with $m$ edges and no isolated vertices. Until recently the
best known bound
for this problem was $2^{c\sqrt{m}\log m}$ (see \cite{AlKrSu}).
To attack Erd\H{o}s' conjecture one can try to use the result on Ramsey numbers
of bounded-degree graphs.
Indeed, given a graph $H$ with $m$ edges, one can first embed the $2\sqrt{m}$
vertices of largest
degree in $H$ using the standard pigeonhole argument of Erd\H{o}s and
Szekeres \cite{ErSz}. The remaining vertices of $H$ span a graph with maximum
degree $\sqrt{m}$. Hence, one may apply the arguments used to prove the upper
bound for Ramsey numbers of bounded-degree graphs to embed the rest of $H$.
However, this approach will likely require an upper bound of
$2^{c\Delta}n$ on the Ramsey number for graphs on $n$
vertices of maximum degree $\Delta$, which we do not have yet.
Recently, the third author \cite{Su} was able to circumvent this difficulty and
prove Erd\H{o}s' conjecture.

Finally, we would like to stress that the proofs given in this paper are highly
specific to the 2-color case. The best results that are known in the $q$-color
case are obtained by an entirely different method \cite{FS07} and are
considerably worse. For example, the $q$-color Ramsey number $r_q(H)$ of a
graph on $n$ vertices with maximum degree $\Delta$ is only known to satisfy the
inequality $r_q(H) \leq 2^{c_q \Delta^2} n$. It would be of considerable
interest to improve this latter bound to $r_q(H) \leq 2^{c_q \Delta^{1+o(1)}}
n$.

\end{document}